\documentclass[preprint, 12pt]{elsarticle}

\usepackage[usenames,dvipsnames]{xcolor}
\usepackage[centertags]{amsmath}
\usepackage{multirow,newlfont,amssymb,amsfonts,amsthm,rotating}

\newcommand{\bb}{\begin{equation}}
\newcommand{\ee}{\end{equation}}

\def\p{\partial}

\def\x1{{\xi }_{xx}}
\def\x2{{\xi }_{yy}}
\def\x3{{\xi }_{xy}}

\def\e1{{\eta }_{xx}}
\def\e2{{\eta }_{yy}}
\def\e3{{\eta }_{xy}}

\def\eps{\varepsilon}

\def\be{\beta}

\def\r{\rho}

\def\l1{{\lambda}_1}

\def\kd{\partial}

\def\bb{\begin{equation}}
\def\ee{\end{equation}}
\def\ba{\begin{array}}
\def\ea{\end{array}}
\def\beqn{\begin{eqnarray}}
\def\eeqn{\end{eqnarray}}

\newcommand{\pd}[1]{ \ensuremath{\frac{\partial}{\partial #1}} }

\newtheorem{thm}{Theorem}[section]
\newtheorem{lem}[thm]{Lemma}

\newtheorem{case}{Case}[subsection]
\newtheorem{subcase}{$\bullet$ Subcase}[case]
\newdefinition{defn}{Definition}[section]

\theoremstyle{remark}
\newtheorem{rem}[defn]{Remark}

\makeatletter
\def\ps@pprintTitle{%
  \let\@oddhead\@empty
  \let\@evenhead\@empty
  \def\@oddfoot{\reset@font\hfil\thepage\hfil}
  \let\@evenfoot\@oddfoot
}
\makeatother

\allowdisplaybreaks[1]

\begin{document} 

%frontmatter

\begin{frontmatter}

%title

\title{Group Classification of a Generalized Black--Scholes--Merton Equation}

%authors

\author[imecc]{Y.~Bozhkov}
%\author{Y.~Bozhkov}
%\address{Instituto de Matem\'atica, Estat\'istica e Computa\c{c}\~ao Cient\'ifica - IMECC\\ 
%Universidade Estadual de Campinas - UNICAMP Rua S\'ergio Buarque de Holanda, $651$\\
%$13083$-$859$ - Campinas - SP, Brasil}
\ead{bozhkov@ime.unicamp.br}
%\ead[url]{http://www.ime.unicamp.br/\~{}bozhkov/}

\author[imecc]{S.~Dimas\corref{cor}}
%\address{Instituto de Matem\'atica, Estat\'istica e Computa\c{c}\~ao Cient\'ifica - IMECC\\ 
%Universidade Estadual de Campinas - UNICAMP Rua S\'ergio Buarque de Holanda, $651$\\
%$13083$-$859$ - Campinas - SP, Brasil}
\ead{spawn@math.upatras.gr}
%\ead[url]{http://www.math.upatras.gr/\~{}spawn}

\cortext[cor]{Corresponding author}

%adresses

\address[imecc]{Instituto de Matem\'atica, Estat\'istica e Computa\c{c}\~ao Cient\'ifica - IMECC\\ 
Universidade Estadual de Campinas - UNICAMP, Rua S\'ergio Buarque de Holanda, $651$\\
$13083$-$859$ - Campinas - SP, Brasil}

%abstract

\begin{abstract}
	The complete group classification of a generalization of the Black--Scholes--Merton model is carried out by making use of the underlying equivalence and additional equivalence transformations.  For each non linear case obtained through this classification, invariant solutions are given.  To that end,  two boundary conditions of financial interest are considered, the terminal and the barrier option conditions.
  \end{abstract}


\end{frontmatter}

\section{Introduction}

In this paper we study the following {\it semilinear} partial differential equation
\begin{equation}\label{u01}
u_t+\frac{1}{2} {\sigma }^2 x^2 u_{xx}+ r x u_x +f(u)=0
\end{equation}
which can be considered as a weakly nonlinear generalization of the celebrated Black--Scholes--Merton equation
\begin{equation}\label{u02}
u_t+\frac{1}{2} {\sigma }^2 x^2 u_{xx}+ r x u_x -r u =0
\end{equation}
that plays a remarkable role in financial Mathematics.  

Usually the terminal condition
\begin{equation}\label{tc}
	u(x,T)=1,
\end{equation}
describing the evolution of standard or ``vanilla"  products; the price of a zero-coupon bond (or of a financial option), $u(x,t)$, which is exercised  when  $t=T$ is considered  \cite{BlaScho72,BlaScho73,Me74,Ugur2k8}.

Furthermore, Eq.~\eqref{u02} may satisfy  also other kinds of options, like the  barrier option.  A barrier option can be considered an exotic option and as such has features that makes it more complex than  the ``vanilla" option, \cite{ha2k11, Kwo2k8, HaSoLea2k13}. The underlying idea is that now a barrier $H(t)$ exists and when the asset price $x$ crosses it, the barrier option $u(x,t)$  becomes extinguished or  comes into existence. Those two types are also known as  \emph{down-and-out} and \emph{down-and-in} respectively. Often a rebate, $R(t)$, is paid if the option is extinguished.  In what follows we shall consider the down-and-out type.

In the context of the Black--Scholes--Merton equation, the  barrier option is expressed by the conditions
\begin{subequations}\label{barrierCond}
	\begin{align}
		&u(H(t),t)=R(t),\label{barrierCond1}\\ 
		&u(x,T)=\max(x-K,0),	
	\end{align}
\end{subequations}
where the barrier option $u(x,t)$ satisfies Eq.~\eqref{u02} for $x>H(t),\ t<T$,  $T$ again is the terminal time where the barrier option is exercised and $K$ is the strike price. A common assumption for the barrier function $H$ is to have the exponential form
\begin{equation}\label{H0}
H(t) = bKe^{-\alpha t},
\end{equation}
where $a\ge0$ and $0\le b\ge1$ \cite[p. 187]{Kwo2k8}.

Although  Eq.~\eqref{u01} is just a mathematical abstraction, it can also be derived from the same stochastic argument as the Black--Scholes--Merton equation \eqref{u02} is derived. In particular, by assuming that the wealth of the portfolio involves a nonlinear function of the value of an option $u(x,t)$ instead of a linear function   and utilizing the same boundary conditions \eqref{tc}, \eqref{barrierCond}. Finally, the constants $r$ and ${\sigma }^2$ represent the risk-free interest rate and the variance of the rate of the return on $u$ respectively. 

The main purpose of this paper is to carry out a complete group classification of a generalized Black--Scholes--Merton equation of type (\ref{u01}).  Recall that to perform a complete group classification of a differential equation (or a system of differential equations) involving arbitrary functions or/and parameters, means to find the Lie point symmetry group $\mathcal{G}$ for the most general case, and then to find specific forms of the differential equation for which $\mathcal{G}$ can be enlarged, \cite[p. 178]{Olver2k}. Quite often, there is good physical or geometrical motivation to study such cases. Moreover for the Black--Scholes--Merton equation \eqref{u02} the motivation is a financial one: for more than three decades this equation has shown its value and gain the trust of the market. But as banks and hedge funds relied more and more on this and its siblings equations they became more and more vulnerable to mistakes or over-simplifications in the mathematics involved in deriving them; it is a linear equation and as such can't follow the dynamic nature and the intricacies  of the modern market.

In \cite{GazIbr98,SinLeaHara2k8,DiAndTsLe2k9} the Black--Scholes--Merton equation (\ref{u02}) as well as other linear evolutionary equations which appear in financial Mathematics have been investigated from the point of view of the S. Lie symmetry theory. Also, the barrier option was studied recently for Eq.~\eqref{u02} under the same prism \cite{HaSoLea2k13}.  The present work can be considered as an extension of that research. As such,  the  linear case $f(u)=\alpha u +\beta$, to avoid a repetition of  already known and established facts, will not be treated explicitly  in the present paper. It will only briefly accounted for the purpose of completeness.

Afterwards, for each nonlinear case of the obtained group classification we look for group invariant solutions taking also into account the boundary conditions \eqref{tc}, \eqref{barrierCond}.

Before we proceed with the group classification, we would  like to remind some facts concerning the group analysis of differential equations.

The symmetry analysis of differential equations is a method first developed in the 19th century by Sophus Lie. One of the main benefits of this method is that by following a completely algorithmic procedure one is able to determine the symmetries of a differential equation or systems of differential equations. \textit{Grosso modo}, the symmetries of a differential equation transform solutions of the equation to other solutions. The Lie point symmetries comprise a structural property of the equation, in essence it is the  DNA of an equation. The knowledge of the symmetries of an equation enables one to utilize them for a variety of purposes, from obtaining analytical solutions and reducing its order to finding of integrating factors and conservation laws. In fact, many, if not all, of the different empirical methods for solving ordinary differential equations (ODEs) we have learned from standard courses at the undergraduate level emerge from a symmetry. For instance, having at our disposal a Lie point symmetry of a first order ODE, we can immediately get explicitly an integrating factor. Furthermore, even the knowledge of a trivial solution of the equation can be used for creating nontrivial solutions by using the equation's symmetries. And all these are due to the rich underlining algebraic structure of the Lie groups and algebras with which we give flesh to the symmetries of a differential equation. 

Another important characteristic of the symmetry method is that in some situations the symmetries of an equation may indicate that it can be transformed to a linear equation. In addition, its symmetries provide the means to construct the needed transformation. A strong indication for that is the existence of an infinite dimensional Lie algebra, \cite{bk}. Recently, it has been shown, by using only the algebraic properties of the symmetries of the equations involved, that the majority  of the well-known differential equations used in economics are linked via an invertible transformation to the heat equation, \cite{DiAndTsLe2k9}. Although for some of the studied equations a transformation between them and the heat equation was already known by other means, the authors of \cite{DiAndTsLe2k9} %were the first to indicate it using 
used the algebraic properties of the symmetries of the equations and followed a straightforward and algorithmic approach.  

In fact, over the last forty years there was a considerable development in the mathematical analysis of partial differential equations which arise in financial Mathematics. 
However when one reads various papers devoted to the resolution of evolution partial differential equations which arise as, more or less, the final stage of the mathematical modeling of some financial process, one can see some \textit{ad hoc} and naive procedures to bring the considered equation under control. A viable alternative way is the employment of symmetry analysis of differential equations.
 
Moreover,  one valuable tool when considering classes of equations is the use of the \emph{equivalence} or \emph{admissible transformations}, \cite{Ovsiannikov82,Ibra2k9,PoEshra2k5}. Equivalence transformations of a class of differential equations are point transformations that keep this class invariant, in other words they map an equation from this class  to another member of the same class. In the recent years equivalence transformations have found much application either as a stand alone analytic  tool for the group classification of differential equations, \cite{RoTo99}, or at the core of the \emph{enhanced group analysis},  \cite{CheSeRa2k8,CaBiPo2k11,IvaPoSo2k10,VaPoSo2k9}. 
 
One of the advantages of this approach, as already emphasized, is that it provides a well defined algorithmic procedure which essentially enables one to find the involved linearizing transformations, conservation laws, invariant solutions, etc. On the other hand, the calculations involved are usually very difficult and extensive even for the simplest equations.  Thus, it may become very tedious and error prone. For this reason the real progress in this area occurred in the last few decades with the advances in computer technology and the development of computer algebra systems like {\it Mathematica, Maple, Reduce}, etc.  Based on these systems, a handful of symbolic packages for determining the symmetries of differential equations exists, \cite{Head93,Nucci1,Nucci2,Baumann2k}. One such symbolic package,  based on {\it Mathematica}, \cite{Wolfram2k10}, has been devised and developed by S. Dimas as part of his PhD thesis, \cite{Dimas2k8}. The package, named SYM, \cite{Dimas2k8,DiTs2k5a,DiTs2k6}, was developed from the ground up using the symbolic manipulation power of  {\it Mathematica} and the artificial intelligence capabilities  which it offers.  It is being used by many researchers around the world. It was extensively used for all the results in the present paper, both for the interactive manipulation of the found symmetries and for the classification of the equations employing the symbolic tools provided by it. 

This paper is organized as follows. In section 2 we present the basic concepts of the Lie point symmetry approach to differential equations used in the paper. In section 3 we obtain the complete group classification of generalized Black--Scholes--Merton equations of type (\ref{u01}). In section 4 we provide invariant solutions for each non-linear case found by the group classification under the two specific boundary problem studied, the ``vanilla" option and the barrier option.  Finally, in section 5 we discuss the obtained results as well as  possible applications.

\section{Preliminaries}

In this section we expose some notions of the modern group analysis that will be encountered in the main sections of the article suitably adapted to the article's needs. For a full treatise in the subject there is a wealth of classical texts that encompass all aspects of the theory, \cite{Olver2k,bk,Ibr85,Ovsiannikov82,Hydon2k,Stephani90}.

A Lie point symmetry of Eq.~\eqref{u01} is a one-parameter transformation of the independent and dependent variables
\begin{align}\label{LG}
	\bar x &= \bar x(x,t,u,\epsilon),\notag\\
	\bar t &= \bar t(x,t,u,\epsilon),\\
	\bar u &= \bar u(x,t,u,\epsilon),\notag
\end{align} 
that keeps \eqref{u01} invariant:
\bb\label{SC}
	\bar u_{\bar t}+\frac{1}{2} {\sigma }^2 \bar x^2 \bar u_{\bar x\bar x}+ r \bar x \bar u_{\bar x} +f(\bar u)=0.
\ee
By substituting in \eqref{SC}  the point transformations \eqref{LG} and using \eqref{u01} we obtain an equation called the \emph{symmetry condition}. From it the exact form of the point transformations is obtained.
However, trying to determine the transformation \eqref{LG}  from  \eqref{SC} is more challenging than trying to solve the Eq.~\eqref{u01} itself! 

The novelty behind S. Lie idea resides on \textit{linearizing} the problem of determining the symmetries of an equation. To do that, the notion of the \textit{infinitesimal generator} is introduced. Namely, the infinitesimal generator is a differential operator 
 $$ X={\xi }^1\frac{\kd }{\kd x}+{\xi }^2\frac{\kd }{\kd t}+ \eta  \frac{\kd }{\kd u}$$
with the functions ${\xi }^i={\xi}^i(x,t,u)$ and $\eta=\eta(x,t,u)$, called infinitesimals, defined as 
 \begin{align}\label{infinitesimals}
 	\xi^1 = \left.\frac{\kd \bar x}{\kd \eps } \right\rvert_{\eps =0}&,& \xi^2 = \left.\frac{\kd \bar t}{\kd \eps } \right\rvert_{\eps =0}&,& \eta = \left.\frac{\kd \bar u}{\kd \eps }\right\rvert_{\eps =0}.
\end{align}
An infinitesimal generator of this type determines a Lie point symmetry of Eq.~\eqref{u01}, if and only if, its action on the equation will be, modulo the equation itself,  identically zero, that is:
\bb\label{lsc}
	\left.X^{(2)}\left[u_t+\frac{1}{2} {\sigma }^2 x^2 u_{xx}+ r x u_x +f(u)\right]\right\rvert_\eqref{u01} \equiv 0,
\ee
where $X^{(2)}$ is the second order prolongation of the operator $X$ given by 
\bb\label{prolong}
	X^{(2)}= X + {\eta }^{(1)}_{i} \frac{\kd }{\kd u_i}+ {\eta }^{(2) }_{i_1 i_2} \frac{\kd }{\kd u_{i_1 i_2}},\ i,i_j=1,2
\ee
and
 \begin{equation}\label{prolongedCoefficients}
 	{\eta }^{(1)}_{i}=D_i\eta - (D_i{\xi }^j)u_j,\,\text{and\ }\, {\eta }^{(2)}_{i_1 i_2} =D_{i_2} {\eta }^{(1)}_{i_1} - (D_{i_2}{\xi }^{j})u_{i_1 j},\ i,i_k,j=1,2
 \end{equation}
where $ u_i = \frac{\kd u}{\kd x^i},\  u_{i_1i_2} = \frac{\kd^2 u}{\kd x^{i_1}\kd x^{i_2}},\ i,i_j=1,2$ and $(x^1,x^2)=(x,t)$, where the Einstein summation convention implied over the indexes. 
From the Eq.~\eqref{lsc}, called \emph{linearized symmetry condition}, an overdetermined system of linear partial differential equations emerges. By solving this system, called the \emph{determining equations}, we find the infinitesimals and hence we obtain the point symmetries of the equation. The group classification occurs in that phase. The determining equations contain also the parameters $\sigma,\r$ and the function $f(u)$. The group classification is performed by investigating each case where specific relations among the unknown elements remove equations from the set of determining equations and hence, expanding the solution space. 

This set of symmetries, represented as differential operators, form a Lie algebra. The system of ODEs \eqref{infinitesimals} with the addition of the conditions $\bar x\rvert_{\eps=0}=x,\ \bar t\rvert_{\eps=0}=t,\ \bar u\rvert_{\eps=0}=u$ forms a well posed initial value problem. By solving it, we can obtain the corresponding local continuous transformations, which form a Lie group. This process is called \emph{exponentiation}. Henceforth, we shall identify a Lie point symmetry with its infinitesimal generator.

Having the symmetries, as a Lie algebra, there is a wealth of things that can be done with. In the present paper, we use them to obtain \emph{invariant} or \textit{similarity solutions} of the Eq.~\eqref{u01}. By invariant solutions we mean solutions of \eqref{u01} that are invariant under one of the found symmetries $\mathfrak{X}$, e.g. 
\begin{equation}\label{isc}
	\left.\mathfrak{X}[u-\varphi (x,t)]\right\rvert_{u=\varphi (x,t)}\equiv0.
\end{equation}
The Eq.~\eqref{isc} is a linear PDE  called \textit{invariant surface condition} and by solving it we obtain a way to reduce Eq.~\eqref{u01}. For example, the symmetry $\pd t$ yields the invariant surface condition $u_t=0$. Solving it, we get the invariant solution $u(x,t)=\phi(x)$ which, in turn, can be used to reduce the order of Eq.~\eqref{u01} effectively turning it from a PDE to an ODE. Similarly, when we look for a similarity solution of Eq.~\eqref{u01} along with a initial/boundary condition we have to choose the subalgebra  leaving also invariant that condition and its boundary:
\begin{equation}\label{tca}
	\left.X(t-T)\right\rvert_{t=T}\equiv0
\end{equation}
and
\begin{equation}\label{tcb}
	\left.X(u-1)\right\rvert_{t=T}\equiv0.
\end{equation}
for the boundary condition \eqref{tc}. And
\begin{equation}\label{boa}
	\left.X(x-H(t))\right\rvert_{x=H(t)}\equiv0
\end{equation}
and
\begin{equation}\label{bob}
	\left.X(u-R(t))\right\rvert_{x=H(t)}\equiv0.
\end{equation}
for the boundary condition \eqref{barrierCond1}.

For obtaining the equivalence transformations there two possible roads. The first one is  by a direct search for the equivalence transformations, an approach that gives in theory the most general equivalence set of transformation, the equivalence group. But it has the same pitfalls as trying to obtain the symmetry group for the equation as briefly discussed already. And the second approach is  through the calculation of the \emph{equivalence algebra} from which the continuous equivalence group can be obtained. In the present work the second road will be realized by complementing  the usual prolongation of the infinitesimal generator with a \emph{secondary prolongation}, \cite{Ibr85}.

To calculate the equivalence algebra, an extension of Eq.~\eqref{u01} must be considered with the arbitrary elements $\sigma, r, f$, now functions of $x,t,u$, and by including the following constraints on them,
$$	\sigma_x=\sigma_t=\sigma_u=r_x=r_t=r_u=f_x=f_t=0.$$
For this extended system the infinitesimal generator is 
\begin{equation} \label{eig}
	X={\xi }^1\frac{\kd }{\kd x}+{\xi }^2\frac{\kd }{\kd t}+ \eta  \frac{\kd }{\kd u}+ \phi^1  \frac{\kd }{\kd \sigma}+ \phi^2  \frac{\kd }{\kd r}+ \phi^3  \frac{\kd }{\kd f},
\end{equation}
 where now the coefficients of this operator depend on the extended space: ${\xi }^i={\xi}^i(x,t,u,\sigma,r,f)$, $\eta=\eta(x,t,u,\sigma,r,f)$ and  $\phi^i=\phi^i(x,t,u,\sigma,r,f)$. 
The second prolongation needed to obtain the determining equation now becomes
\begin{multline}\label{prolong2}
	X^{(2)}= X + {\eta }^{(1)}_{i} \frac{\kd }{\kd u_i}+ {\eta }^{(2) }_{i_1 i_2} \frac{\kd }{\kd u_{i_1 i_2}}+ {\phi}^{1,(1)}_{j} \frac{\kd }{\kd \sigma_j}+ {\phi}^{2,(1)}_{j} \frac{\kd }{\kd r_j}\\
	+ {\phi}^{3,(1)}_{j} \frac{\kd }{\kd f_j},\ i,i_k=1,2,\ j=1,2,3
\end{multline}
where the coefficients $ {\eta }^{(1)}_{i}, {\eta }^{(2) }_{i_1 i_2}$ are calculated as usual by the formula \eqref{prolongedCoefficients}  while for the coefficients ${\phi}^{i,(1)}_{j} $ with the secondary prolongation,
$$
	{\phi }^{i, (1)}_{j}=\tilde D_j\phi^i - (\tilde D_j{\xi }^1)p^i_x- (\tilde D_j{\xi }^2)p^i_t- (\tilde D_j\eta )p^i_u,\ , i,j=1,2,3
$$
where $(p^1,p^2,p^3 )=(\sigma,r,f)$, $(x^1,x^2,x^3)=(x,t,u)$ and
$$
	\tilde D_j = \frac{\kd}{\kd x^j}+ p^i_{x^j} \frac{\kd }{\kd p^i}, 
$$
again the Einstein summation convention is used for the index $i$. After that point we follow Lie's algorithm as usual.

Having the equivalence algebra by exponentiation one can obtain the continuous part of the equivalence group. By using the method proposed in \cite[pp. 187 c.f.]{Hydon2k}, \cite{Hy2k} one can also obtain the discrete part and hence retrieve the whole set of equivalence transformations permissible by this class of equations. 

Another useful notion is that of the \emph{additional equivalence transformation}. An additional equivalence transformation is a point transformation that connects inequivalent classes of differential equations. The knowledge of such transformations greatly helps the classification.

\section{Group classification}

In this section we proceed with the group classification of the generalized Black--Scholes--Merton equation \eqref{u01}. 

First, the best representative for the class of equations \eqref{u01} is obtained utilizing its equivalence algebra. To do that, the continuous part of the equivalence group is constructed and with its help as many as possible arbitrary elements are zeroed.
\begin{thm}
The equivalence algebra $\hat{\mathcal{L}}_\mathcal{E}$ of  class \eqref{u01} is generated by the following vector fields
\begin{gather*}
	\partial _t,\ \partial _u,\ x\partial _x,\ \partial _r+t x\partial _x,\ f\partial _f+u\partial _u,\\
 	x \left(2 r t-t \sigma ^2+2 \log \lvert x\rvert\right)\partial _x+4 t\partial _t-4 f\partial _f, \\
 	x \left(\left(2 r+\sigma ^2\right)t-2 \log\lvert x\rvert\right)\partial _x -2 \sigma \partial _{\sigma }.
\end{gather*}
\end{thm}
\begin{proof}
By applying the second order prolongation \eqref{prolong2} of the infinitesimal generator \eqref{eig} to the extended system 
\begin{gather*}
	u_t+\frac{1}{2} {\sigma }(x,t,u)^2 x^2 u_{xx}+ r(x,t,u) x u_x +f(x,t,u)=0,\\
	\sigma_x=\sigma_t=\sigma_u=r_x=r_t=r_u=f_x=f_t=0,
\end{gather*}
modulo the extended system itself, we get the system of determining equations:
\begin{gather*}
{\eta_3}_{f}=0,\ {\eta_4}_{f}=0,\ \xi ^2_{f}=0,\ \xi ^2_{f}=0,\ \xi ^2_{f}=0,\  \xi ^2_{ff}=0,\  {\eta_3}_{u}=0,\  {\eta_4}_{u}=0,\\
\xi ^2_{u}=0,\ \xi ^2_{uf}=0,\  \xi ^2_{uu}=0,\  {\eta_1}_{t}=0,\  {\eta_2}_{t}=0,\  {\eta_3}_{t}=0,\  {\eta_4}_{t}=0,\  {\eta_1}_{x}=0,\\  
 {\eta_2}_{x}=0,\  {\eta_3}_{x}=0,\ {\eta_4}_{x}=0,\  \xi ^2_{x}=0,\ {\eta_1}_{f}+f \xi ^2_{f}=0,\ {\eta_1}_{f}+f \xi ^2_{f}=0,\\ 
 \xi ^1_{f}-r x \xi ^2_{f}=0,\  \xi ^1_{ff}-r x \xi ^2_{ff}=0,\ \xi ^1_{uf}-r x \xi ^2_{uf}=0,\  \xi ^1_{uu}-r x \xi ^2_{uu}=0,\\
2 \xi ^2_{f}+{\eta_1}_{ff}+f \xi ^2_{ff}=0,\ 2 \xi ^1_{f}-x \left(2 \left(r+\sigma ^2\right) \xi ^2_{f}+x \sigma ^2 \xi ^2_{xf}\right)=0,\\
2 \xi ^1_{u}-x \left(2 \left(r+\sigma ^2\right) \xi ^2_{u}+x \sigma ^2 \xi ^2_{xu}\right)=0,\\
2 r \xi ^2_{u}+{\eta_1}_{uu}+f \xi ^2_{uu}-2 \xi ^1_{xu}+2 r x \xi ^2_{xu}=0, \\
r \xi ^2_{f}+\xi ^2_{u}+{\eta_1}_{uf}+f \xi ^2_{uf}-\xi ^1_{xf}+r x \xi ^2_{xf}=0,\\ 
f \xi ^1_{f}+x \left(-f r \xi ^2_{f}+x \sigma ^2 \left(\xi ^2_{x}+{\eta_1}_{xf}+f \xi ^2_{xf}\right)\right)=0,\\
  4 x {\eta_3}+\sigma  \left(4 \xi ^1+x \left(-2 f \xi ^2_{u}+2 \xi ^2_{t}-4 \xi ^1_{x}+6 r x \xi ^2_{x}+4 x \sigma ^2 \xi ^2_{x}+x^2 \sigma ^2 \xi ^2_{xx}\right)\right)=0, \\
\begin{split}
	2 {\eta_2}-2 f {\eta_1}_{u}-2 f^2 \xi ^2_{u}+2 {\eta_1}_{t}+2 f \xi ^2_{t}+2 r x {\eta_1}_{x}+2 f r x \xi ^2_{x}\\
	+\sigma ^2  x^2( {\eta_1}_{xx}+f \xi ^2_{xx})=0
\end{split}\\
\begin{split}
 	2 x {\eta_4}+2 r \xi ^1+2 f \xi ^1_{u}-2 f r x \xi ^2_{u}-2 \xi ^1_{t}+2 r x \xi ^2_{t}-2 r x \xi ^1_{x}+2 r^2 x^2 \xi ^2_{x}+2 r x^2 \sigma ^2 \xi ^2_{x}\\
 		+2 x^2 \sigma ^2 {\eta_1}_{xu}+2 f x^2 \sigma ^2 \xi ^2_{xu}-x^2 \sigma ^2 \xi ^1_{xx}+r x^3 \sigma ^2 \xi ^2_{xx}=0.
\end{split}
\end{gather*}
Solving the above system the equivalence algebra $\hat{\mathcal{L}}_\mathcal{E}$ is obtained.
\end{proof}
\begin{lem}
	The continuous part of the equivalence group, $\hat{\mathcal{E}}_\mathcal{C}$, consists of the transformations
	\begin{align*}
		\tilde x &= e^{\frac{1}{2} t \delta _6 \left(\left(\sigma ^2-2 r\right) \delta _7-\sigma ^2 \delta _7^2 \delta _6+2 \delta _6 (r+\delta_5)\right)} \lvert x\rvert^{\delta _6 \delta _7} \delta _4,\\
		\tilde t &= \delta _1+ \delta _6^2t,\quad\tilde u = \delta _2+ \delta _3u,\quad\tilde r = r+\delta _5,\quad\tilde\sigma =  \delta _7\sigma ,\quad\tilde f = \frac{ \delta _3}{\delta _6^2}f,
	\end{align*}
	where $\delta_i$ are arbitrary constants and $\delta_3,\delta_4,\delta_6,\delta_7\ne0$.
\end{lem}
\begin{rem}
	Due to the fact that the transformation for $x$ depends also on the arbitrary elements of Eq.~\eqref{u01},  $\hat{\mathcal{E}}_\mathcal{C}$ is also called the continuous part of the \emph{generalized equivalence group}. If one has chosen to assume that the equivalence transformations for $x,t,u$ do not depend also on the arbitrary elements $r,\sigma,f$, i.e. $\xi^i = \xi^i(x,t,u),\ \eta=\eta(x,t,u)$ in \eqref{eig}, then  the continuous part of the \emph{usual equivalence group} $\mathcal{E}_\mathcal{C}$ would be obtained. 
\end{rem}
Readily,  using $\hat{\mathcal{E}}_\mathcal{C}$ one can find an equivalence transformation that $\tilde\sigma\rightarrow\sqrt{2}$ and $\tilde r\rightarrow0$:
\begin{equation}\label{equivTrans}
		\tilde x = e^{\left(\frac{\sigma ^2-2 r}{\sqrt{2} \sigma }-1\right)t} \lvert x\rvert^{\frac{\sqrt{2}}{\sigma }},\quad\tilde t = t,\quad\tilde u = u,\quad\tilde r = 0,\quad\tilde\sigma = \sqrt2,\quad\tilde f = f.
\end{equation}
Using \eqref{equivTrans},	Eq.~\eqref{u01} turns into the equation
\begin{equation}\label{u01a}
	\tilde u_{\tilde t}+ \tilde x^2 \tilde u_{\tilde x\tilde x} +\tilde f(\tilde u)=0.
\end{equation}
Next, using the additional equivalence transformation 
\begin{equation}
	\hat x = \log\lvert \tilde x\rvert,\quad \hat t=t,\quad\hat u = \lvert\tilde x\rvert^{-1/2} \tilde u,
\end{equation}
the heat equation with a nonlinear source is obtained (for clarity henceforth the hats are dropped)
\begin{equation}\label{heat}
	 u_{ t}+ u_{ x x}+e^{- x/2}f(e^{ x/2} u)-\frac{1}{4} u=0.
\end{equation}
Therefore, without any loss of generality, Eq.~\eqref{heat} will be classified instead. In addition, the terminal condition \eqref{tc} is transformed to the condition
\begin{equation}\label{tc2}
	u(x,T) = \exp(-x/2)
\end{equation}
and for the barrier option condition \eqref{barrierCond1} we have:
\begin{equation}\label{barrierCond1a}
	e^{\frac{1}{2}\left(\frac{\sigma^2-2r}{\sqrt{2}\sigma}t-1\right)}\lvert H(t)\rvert^\frac{\sqrt{2}}{2\sigma}u\left(\left(\frac{\sigma^2-2r}{\sqrt{2}\sigma}-1 \right)t+\frac{\log\lvert H(t) \rvert}{\sqrt{2}\sigma},t\right)=R(t)
\end{equation}
Finally, the equivalence group for Eq.~\eqref{heat} is given by the following theorem
\begin{thm}
	The  equivalence group, $\mathcal{E}$, of Eq.~\eqref{heat} consists of the transformations
	\begin{align*}
		\tilde x &=\delta_5\beta x + (\beta-\delta_5)\delta_5 t + 2 \delta_4,\quad\tilde t = \delta _5 t+\delta _1,\\
		\tilde u &=e^{-\frac{1}{2}\left(\delta_5 x-(\delta_5-1)\delta_5t+2\delta _4\right)}\left(\alpha \delta_3 e^{\frac{x-(\beta-1)(x+t)}{2}} u+\delta_2 \right),\quad\tilde f =\alpha \frac{ \delta _3}{\delta _5^2}f,
	\end{align*}
	where $\delta_i$ are arbitrary constants, $\delta_3,\delta_5\ne0$ and $\alpha,\beta=\pm1$.
\end{thm}
\begin{proof}
	The process is analogous to the one for Eq.~\eqref{u01} with the only difference that now we have only one arbitrary element, the function $f$. In addition, using the process described in \cite{Hy2k} we find the four discrete equivalence transformations
	$$
		(x,t,u,f)\rightarrow(\beta x+(\beta-1)t,t,\alpha e^{-\frac{1}{2}(\beta-1)(x+t)} u,\alpha f),
	$$
	where $\alpha,\beta=\pm1$.
Together the two sets of transformations comprise the usual equivalence group of transformations $\mathcal{E}$.
\end{proof}

Equation \eqref{heat} belongs to the class
\begin{equation*}
	u_t = u_{xx} + f(t,x,u,u_x)
\end{equation*}
that describes nonlinear heat conductivity processes. Due to the fact that the above class was completely classified in \cite{ZhdaLa99}. Apart from mentioning the classification equation
\begin{multline*}
	\frac{1}{2} \left(u f^\prime\left(e^{x/2} u\right)-e^{-x/2} f\left(e^{x/2} u\right)\right) \left(\mathcal{F}_3(t)+\frac{x \mathcal{F}_2^\prime(t)}{2}\right)+u \mathcal{F}^\prime_4(t)\\
	+\left(e^{-x/2} f\left(e^{x/2} u\right)-\frac{u}{4}\right) \left(\mathcal{F}^\prime_2(t)-\mathcal{F}_4(t)-\frac{1}{8} x \left(4 \mathcal{F}^\prime_3(t)+x \mathcal{F}_2^{\prime \prime}(t)\right)\right)\\
	+\left(f^\prime\left(e^{x/2} u\right)-\frac{1}{4}\right) \left(\mathcal{F}_1(x,t)+\frac{1}{8} u \left(8 \mathcal{F}_4(t)+x \left(4 \mathcal{F}_3^\prime(t)+x \mathcal{F}_2{}^{\prime \prime }(t)\right)\right)\right)\\
	+\frac{1}{8} u \left(2 \mathcal{F}_2^{\prime \prime }(t)+x \left(4 \mathcal{F}_3^{\prime \prime }(t)+x \mathcal{F}_2^{\prime \prime \prime}(t)\right)\right)+{\mathcal{F}_1}_{t}(x,t)+{\mathcal{F}_1}_{xx}(x,t)=0
\end{multline*}
no further details of the calculations involved will be showed. We proceed with presenting the resulting classification.

\begin{enumerate}
	\item For an arbitrary $f$ the Lie point symmetries of (\ref{heat}) are determined by the  infinitesimal generators:
		\begin{align}\label{u03}
			\mathfrak{X}_1&=\frac{\p}{\p t},& \mathfrak{X}_2&=2 \frac{\p}{\p x}-u\frac{\p}{\p u}. 
		\end{align}
	\item For $f(\zeta) =-\frac{\gamma}{\beta}(\alpha +\beta \zeta)\left(\delta +\log\lvert\alpha+\beta\zeta\rvert\right),\,\be,\gamma\neq 0$,  in addition to the symmetries \eqref{u03} we get  also the symmetries		
		\begin{subequations}\label{u03a}
		\begin{align}
			\mathfrak{X}_3 &= \frac{\alpha+\beta e^{x/2}u}{\beta}e^{-\frac{x}{2}+\gamma t}\frac{\p}{\p u},\\
			\mathfrak{X}_4 &=  \frac{2}{\gamma}e^{\gamma t}\frac{\p}{\p x} +  \frac{\beta  \gamma (t+x)e^{x/2}u + \alpha  (1+\gamma (t+x)  )}{\beta\gamma}e^{-\frac{x}{2}+\gamma t}\frac{\p}{\p u}.			
		\end{align}
		\end{subequations}
	\item For $f(\zeta)=\alpha(\zeta-\beta)^2,\,\alpha\ne0$, in addition to the symmetries \eqref{u03} we have the symmetry
		\begin{align}\label{u03b}
			X_3&=2(x-t)\pd{x}+4t\pd{t}+((t-x-4)u+4\beta e^{-x/2})\pd{u}.
		\end{align}
	\item For completeness, we present also the additional symmetries for the linear cases:  
		\begin{itemize}
			\item $f(\zeta)=\beta \zeta+\alpha,\, \beta\ne0$:
				\begin{align*}
					X_3&=(u+\frac{\alpha}{\beta}e^{-x/2})\pd u,\\ 
					X_4&=2t\pd x+\left(u x+\frac{ \alpha(t+x)e^{-x/2}  }{\beta }\right)\pd u,\\
					X_5&=2x\pd x+4t\pd t+\frac{\left(\alpha x- (4 \beta -1) \left(\alpha +\beta e^{x/2} u  \right)t\right)}{\beta }e^{-x/2} \pd u,\\
					X_6&=4xt\pd x+4t^2\pd t+\frac{1}{\beta}\left(\alpha  \left(2 (x-1)t+x^2+(1-4 \beta )t^2 \right)e^{-x/2}\right.\\
					&\left.+\beta  \left(x^2+(t-4\beta t  -2)\right)tu\right)\pd u,\\
					X_\infty&= \mathcal{F}(x,t)\pd u,
				\end{align*}
				where the smooth function $ \mathcal{F} =\mathcal{F}(x,t)$  satisfies the linear equation $$\mathcal{F}_{t}+ \mathcal{F}_{xx}+(\alpha-\frac{1}{4})\mathcal{F}=0$$
			and,
			\item $f(\zeta)=\alpha$:
				\begin{align*}
					X_3&=(u+\alpha x e^{-x/2})\pd u,\\ 
					X_4&=4t\pd x+\left(2 e^{x/2} x u+\alpha\left(t^2-x (x+2)\right)  \right)e^{-x/2} \pd u,\\
					X_5&=4x\pd x+8t\pd t+  \left(2 e^{x/2} t u+\alpha\left(t^2-(x-6) x\right)  \right)e^{-x/2}\pd u,\\
					X_6&=12xt\pd x+12t^2\pd t+\left(3 \left((t-2) t+x^2\right)u\right.\\
					&\left.+\alpha \left(t^3-12 t^2-3 t x^2-2 x (x (x+3)+6)\right)e^{-x/2} \right)\pd u,\\
					X_\infty&= \mathcal{F}(x,t)\pd u,
				\end{align*}
		\end{itemize}	
		where the smooth function $ \mathcal{F} =\mathcal{F}(x,t)$  satisfies the linear equation $$\mathcal{F}_{t}+ \mathcal{F}_{xx}+(\alpha-\frac{1}{4})\mathcal{F}=0.$$
\end{enumerate} 
\begin{rem}
Looking at the corresponding Lie algebars for the two linear cases, it is evident  that they are linked to the heat equation via a point transformation, \cite{DiAndTsLe2k9},  --- an additional equivalence transformation --- a fact well established in the literature, \cite{GazIbr98,SinLeaHara2k8}.
\end{rem}

\section{Invariant solutions}

Having obtained the complete group classification for Eq.~\eqref{heat}, and consequently for Eq.~\eqref{u01}, we can look for invariant solutions under the terminal condition \eqref{tc2} and the barrier option \eqref{barrierCond1a}: For each one of the two nonlinear cases the appropriate sub algebra of symmetries also admitted by each problem is found using the two required conditions  \eqref{tca}, \eqref{tcb} and \eqref{boa}, \eqref{bob} adapted now to Eq.~\eqref{heat}. Finally, by using the subalgebra obtained for every subcase that surfaced from the two conditions a similarity solution is constructed.

\subsection{The terminal condition}

\begin{case}$f(\zeta)=\alpha(\zeta-\beta)^2,\,\alpha\ne0$
\end{case}
Let the arbitrary element of the Lie algebra spanned by \eqref{u03} and \eqref{u03b} be $\mathfrak X = c_1 \mathfrak X_1+c_2 \mathfrak X_2+c_3 \mathfrak X_3$. Using \eqref{tca}, \eqref{tcb} and \eqref{tc2} we obtain the conditions: 
$$
	c_1=-4Tc_3
$$
and
$$
	(\beta-1)c_3=0.
$$
From the above conditions two specific subcases  occur:
\begin{subcase}
	$c_1=c_3=0,\ c_2\ne0$
\end{subcase}
For this case the only symmetry that keeps invariant both Eq.~\eqref{heat} and \eqref{tc2} is the $ \mathfrak{X}_2=2 \partial_x-u\partial_u$.
\begin{subcase}
	$\beta=1,\ c_1=-4Tc_3$
\end{subcase}
For this case the only symmetries that keeps invariant both Eq.~\eqref{heat} and \eqref{tc2} are 
\begin{align*}
	Z_1&=2 \partial_x-u\partial_u,\\
	Z_2&=2(x-t)\partial_x+4(t-T)\partial_t+((t-x-4)u+4 e^{-x/2})\partial_u
\end{align*}
\begin{case} $f(\zeta) =-\frac{\gamma}{\beta}(\alpha +\beta \zeta)\left(\delta +\log\lvert\alpha+\beta\zeta\rvert\right),\,\be,\gamma\neq 0$
\end{case}
Let the arbitrary element of the Lie algebra spanned by \eqref{u03} and \eqref{u03a} be $\mathfrak X = c_1 \mathfrak X_1+c_2 \mathfrak X_2+c_3 \mathfrak X_3+c_4\mathfrak X_4$. Using \eqref{tca}, \eqref{tcb} and \eqref{tc2} the conditions are 
$$
	c_1=0
$$
and
$$
	(\alpha +\beta) \left(c_4+\gamma  \left(c_3+(t+x)c_4\right)\right)=0.
$$
From the above conditions two specific subcases  occur:
\begin{subcase}
	$c_1=c_3=c_4=0,\ c_2\ne0$
\end{subcase}
For this case the only symmetry that keeps invariant both Eq.~\eqref{heat} and \eqref{tc2} is the $ \mathfrak{X}_2=2 \partial_x-u\partial_u$.
\begin{subcase}
	$\beta=-\alpha\ne0,\ \gamma\ne0,\ c_1=0$
\end{subcase}
For this case the only symmetries that keeps invariant both Eq.~\eqref{heat} and \eqref{tc2} are 
\begin{align*}
	Z_1&=2 \partial_x-u\partial_u,\\
	Z_2&=(1-e^{x/2}u)e^{-\frac{x}{2}+\gamma t}\partial_u\\
	Z_3&= \frac{2}{\gamma}e^{\gamma t}\partial_x +  \frac{\gamma (t+x)e^{x/2}u -  (1+\gamma (t+x)  )}{\gamma}e^{-\frac{x}{2}+\gamma t}\partial_u
\end{align*}

The common denominator for all the above cases is the symmetry $2 \partial_x-u\partial_u$. This symmetry gives the invariant solution
$$
	u(x,t)= C(t) e^{-x/2}.
$$
Going back to the Eq.~\eqref{u01} we see that it corresponds to the \emph{trivial} assumption 
$$
	u(x,t) = C(t),
$$
i.e. solutions that do not have dependency on the variable $x$.

\subsection{The condition for barrier option}

\begin{case}$f(\zeta)=\alpha(\zeta-\beta)^2,\,\alpha\ne0$
\end{case}
Let the arbitrary element of the Lie algebra spanned by \eqref{u03} and \eqref{u03b}, $\mathfrak X = c_1 \mathfrak X_1+c_2 \mathfrak X_2+c_3 \mathfrak X_3$. Using \eqref{boa} and the equivalence transformations the first condition turns to the ODE:%\eqref{barrierCond1a} the conditions are 
\begin{multline}\label{ode1}
	\left(2 \sqrt{2} r \left(c_1+2 t c_3\right)+\sigma  \left(\left(2-\sqrt{2} \sigma \right) c_1+4 c_2-2 \sqrt{2} t\sigma  c_3t\right)\right)H\\
		+4 \sqrt{2} c_3 H\log\lvert H\rvert   	-2 \sqrt{2} \left(c_1+4 t c_3\right) H^\prime=0
\end{multline}
two cases are discerned,  $c_3\ne0$ and $c_3=0$:
\begin{subcase}
	$c_3\ne0$
\end{subcase}
The solution of \eqref{ode1} is
\begin{equation}\label{H1a}
	H(t)=e^{-\frac{\sigma}{\sqrt{2} }\lambda+(r- \frac{1}{2} \sigma ^2)\left(t+\mathcal A \sqrt{\kappa+t}\right)}
\end{equation}
where $\kappa = \frac{c_1}{4c_3},\ \lambda=\frac{c_1+2c_2}{2c_3}$ and $\mathcal A$ the constant of integration. Using this solution, conditions \eqref{bob} and \eqref{barrierCond1a} we get the ODE
\begin{equation}\label{ode2}
	\beta -R-\left(\kappa+t \right) R^\prime=0.
\end{equation}
Its solution is
\begin{equation}\label{R1a}
	R(t)=\frac{\mathcal B+ \beta t}{\kappa+t},
\end{equation}
where $\mathcal B$ the constant of integration. 

Having found the functions $H,R$ admitted by the symmetries we proceed with the reduction of the Eq.~\eqref{heat}. From the invariant surface condition \eqref{isc} the invariant solution is
\begin{equation}\label{u1a}
	u(x,t)=\frac{e^{-\frac{1}{2} \left(x+\lambda\right)} \left(4 e^{\frac{\lambda}{2}} \beta t +F\left(\frac{x + t+\lambda}{\sqrt{\kappa +t}}\right)\right)}{4(\kappa +t)}.
\end{equation}
Substituting \eqref{u1a} to \eqref{heat} for this particular case for $f$ we arrive at the reduction
\begin{equation*}
	16 e^{\frac{1}{2}  \lambda} \beta  \kappa  (\beta \kappa +\frac{1}{ \alpha} )-8( \frac{1}{ 2\alpha}+ \beta  \kappa ) F(\zeta)+e^{-\frac{1}{2}  \lambda} F(\zeta)^2-\frac{2}{ \alpha}  \zeta  F^\prime+\frac{1}{ \alpha}  F^{\prime \prime }=0,
\end{equation*} 
where $\zeta =  \frac{x + t+\lambda}{\sqrt{\kappa +t}}$. Although the general solution of this equation cannot be found in a closed form, it has the special solution
$$
	F(\zeta) = 4e^{\frac{1}{2} \lambda } \left(\beta  \kappa-\frac{3 }{2\alpha \zeta ^2}\right).
$$ 
Using it in combination with \eqref{u1a} we arrive to the invariant solution 
$$
u(x,t) = \frac{ \left(\beta  (t+x+\lambda)^2-\frac{6}{\alpha}\right)}{(t+x+\lambda )^2}e^{-x/2}.
$$
Finally, using the equivalence transformations and the boundary condition \eqref{barrierCond1} we arrive to the similarity solution for Eq.~\eqref{u01} with $f(\zeta)=\alpha(\zeta-\beta)^2$
$$
	u(x,t) = \beta -\frac{24 \sigma ^2}{\alpha \left(2 \lambda  \sigma +\sqrt{2}( \sigma ^2-2 r)t+2 \sqrt{2} \log\lvert x\rvert\right)^2}
$$
with
$$
	H(t) = e^{-\frac{\sigma}{\sqrt{2} }\lambda+(r- \frac{1}{2} \sigma ^2)\left(t+\mathcal A \sqrt{\kappa+t}\right)}
$$
and
$$
	R(t)=\beta -\frac{12   \sigma ^2}{\mathcal A^2 \alpha (t+\kappa ) \left(\sigma ^2-2r\right)^2}
$$
\begin{subcase}
	$c_3=0$
\end{subcase}
For $c_3=0$ the solution of \eqref{ode1} becomes
\begin{equation}\label{H1b}
	H(t)=\mathcal Ae^{\frac{1}{2}(2r - \sigma^2)t+  \lambda t} 
\end{equation}
where $\lambda=\frac{\left(\mathbf{c}_1+2 \mathbf{c}_2\right)}{\sqrt{2}\mathbf{c}_1}\sigma$\footnote{$c_1\ne0$ otherwise $c_2$ must also be zero.} and $\mathcal A\ne0$ the constant of integration. Using this solution, conditions \eqref{bob} and \eqref{barrierCond1a} we get the ODE
\begin{equation}\label{ode3}
	c_1 R^\prime=0.
\end{equation}
Hence
\begin{equation}\label{R1b}
	R(t)= \mathcal B,
\end{equation}
where $\mathcal B$ is constant. 

Having found the functions $H,R$ admitted by the symmetries we proceed with the reduction of the Eq.~\eqref{heat}. From the invariant surface condition \eqref{isc} the invariant solution is
\begin{equation}\label{u1b}
	u(x,t)=e^{-x/2} F\left(\frac{x+t \left(1-\sqrt{2} \lambda \right)}{1-\sqrt{2} \lambda }\right).
\end{equation}
Substituting \eqref{u1b} to \eqref{heat} for this particular case for $f$ we arrive to the reduction
\begin{equation*}
	\left(\sqrt{2} \lambda -1\right) \left(\alpha\left(1-\sqrt{2} \lambda \right) (\beta -F(\zeta))^2-\sqrt{2}  \lambda  F^\prime\right)- F^{\prime \prime }=0,
\end{equation*} 
where $\zeta = \frac{x+t \left(1-\sqrt{2} \lambda \right)}{1-\sqrt{2} \lambda }$. Although this equation cannot be analytically solved it has for $\lambda=0$\footnote{Actually for $\lambda=0$ its general solution can be given but  in implicit form containing transcendental functions, hence not suitable for our analysis.} the special solution 
$$
	F(\zeta) =\beta-\frac{6\alpha }{\alpha\zeta^2} .
$$ 
Using it in combination with \eqref{u1b} we arrive at the invariant solution 
$$
u(x,t) = e^{-x/2} \left(\beta-\frac{6}{\alpha(t+x)^2} \right).
$$
Finally, using the equivalence transformations and the boundary condition \eqref{barrierCond1} we obtain the similarity solution for Eq.~\eqref{u01} with $f(\zeta)=\alpha(\zeta-\beta)^2$
$$
	u(x,t) = \beta -\frac{12 \sigma ^2}{ \alpha \left(t \left(\sigma ^2-2 r\right)+2 \log\lvert x\rvert\right)^2}
$$
with
$$
	H(t)=\mathcal Ae^{\frac{1}{2}(2r - \sigma^2)t} 
$$
and
$$
	R(t)=\beta -\frac{3  \sigma ^2}{\alpha \log^2\lvert\mathcal A\rvert}
$$

\begin{case} $f(\zeta) =-\frac{\gamma}{\beta}(\alpha +\beta \zeta)\left(\delta +\log\lvert\alpha+\beta\zeta\rvert\right),\,\be,\gamma\neq 0$
\end{case}
Let the arbitrary element of the Lie algebra spanned by \eqref{u03} and \eqref{u03a} be $\mathfrak X = c_1 \mathfrak X_1+c_2 \mathfrak X_2+c_3 \mathfrak X_3+c_4\mathfrak X_4$. Using \eqref{boa} and the equivalence transformations the first condition turns to the ODE: 
\begin{equation}\label{ode4}
	2 \left(c_3+\frac{c_4}{\gamma}e^{t \gamma }\right)+c_1 \left(1+\frac{\sqrt{2} r}{\sigma }-\frac{\sigma }{\sqrt{2}}-\frac{\sqrt{2} H'}{\sigma  H}\right)=0,
\end{equation}
where $c_1\ne0$.
The solution of \eqref{ode4} is
\begin{equation}\label{H2a}
	H(t)=\mathcal A e^{\frac{1}{2}(2r -\sigma^2)t+\frac{1}{2} \sigma  \left(\lambda t+\mu e^{\gamma t}\right)}, 
\end{equation}
where $\mu = \frac{2\sqrt{2}c_4}{\gamma^2c_1},\ \lambda=\sqrt{2}(1+2\frac{c_3}{c_1})$ and $\mathcal A$ the constant of integration.  Using the above solution, conditions \eqref{bob} and \eqref{barrierCond1a} we get the ODE
\begin{multline}\label{ode5}
	e^{t \gamma } \left(\kappa +2 \gamma  \mu  \left(\sqrt{2} \sigma+\sigma\gamma  \lambda t +2 \gamma  \log\mathcal A\right)\right) (\alpha +\beta  R)\\
	+2 e^{2 \gamma t } \gamma ^2 \mu ^2 \sigma  (\alpha +\beta  R)-8 \beta  \sigma  R'=0.
\end{multline}
Its solution is
\begin{equation}\label{R2a}
	R(t)=-\frac{\alpha }{\beta }+\mathcal Be^{\frac{\left(\kappa +\gamma  \mu  \left(2 \sqrt{2}-2 \lambda +2 \gamma  \lambda t +\gamma  \mu e^{t \gamma }  \right) \sigma +4 \gamma ^2 \mu  \log\mathcal A\right)}{8 \gamma  \sigma }e^{\gamma t} } 
\end{equation}
where  $\kappa = 8\frac{c_2}{\sigma c_1}$ and $\mathcal B$ the constant of integration. 

Having found the functions $H,R$ admitted by the symmetries we proceed with the reduction of the Eq.~\eqref{heat}. From the invariant surface condition \eqref{isc} the invariant solution is
\begin{multline}\label{u2a}
	u(x,t)=-\frac{ \alpha }{\beta }e^{-x/2}+e^{\frac{1}{8} \left(\frac{\left(2 \gamma  \left(\sqrt{2} ) \gamma(t+x) -\lambda \right) \mu +\kappa  \sigma \right)}{\gamma }e^{t \gamma } -2  \left(\sqrt{2} \lambda -2\right)t-\gamma  \mu ^2e^{2 t \gamma } \right)}\\
		 F\left(t+x-\frac{ \lambda }{\sqrt{2}}t-\frac{\mu }{\sqrt{2}}e^{t \gamma } \right).
\end{multline}
Substituting \eqref{u2a} to \eqref{heat} for this particular case for $f$ we arrive at the reduction
\begin{multline*}
	F(\zeta ) \left(2 \gamma  (2 \delta +\zeta )-1+\sqrt{2} \lambda +4 \gamma  (\log\beta +\log F(\zeta ))\right)+2 \left(\sqrt{2} \lambda -2\right) F^\prime\\
		-4 F^{\prime \prime }=0,
\end{multline*} 
where $\zeta = t+x-\frac{ \lambda }{\sqrt{2}}t-\frac{\mu }{\sqrt{2}}e^{t \gamma }$. Although for this equation the general solutions cannot be found in a closed form,  two particular solutions can be obtained, one when $\lambda\ne0$ and one when $\lambda=0$. Each of them is considered separately:
\begin{itemize}
	\item	$\lambda\ne0$
		\begin{equation}\label{specialSol1}
			F(\zeta) = \Delta_1 e^{ -\frac{1}{2} \zeta -\log\beta}
		\end{equation}
		with $\Delta_1=e^{-\delta}$. Using it in combination with \eqref{u2a} we arrive at the invariant solution 
		$$
			u(x,t) = \frac{e^{-\frac{x}{2}} }{\beta}\left(\Delta_1e^{-\frac{\gamma ^2 \mu ^2e^{2  \gamma t} -e^{ \gamma t} \left(2 \gamma  \left(\sqrt{2}+\sqrt{2} (t+x) \gamma -\lambda \right) \mu +\kappa  \sigma \right)}{8 \gamma }}- \alpha \right).
		$$
		Again, using the equivalence transformations and the boundary condition \eqref{barrierCond1} we arrive to the similarity solution for Eq.~\eqref{u01} with $f(\zeta)=-\frac{\gamma}{\beta}(\alpha +\beta \zeta)\left(\delta +\log\lvert\alpha+\beta\zeta\rvert\right)$
		$$
			u(x,t) = \frac{\Delta_1}{\beta}e^{\frac{e^{\gamma t} \left(\sigma  \left(\kappa  \sigma +\gamma  \mu  \left(2 \sqrt{2}-2 \lambda -\gamma  \mu e^{\gamma t} +2  \gamma  \sigma t\right)\right)-4 r \gamma ^2 \mu t\right)}{8 \gamma  \sigma }}\lvert x\rvert^{\frac{\gamma\mu}{2\sigma}e^{\gamma t}}-\frac{\alpha }{\beta }
		$$
		with
		$$
			H(t) = \mathcal A e^{\frac{1}{2}(2r -\sigma^2)t+\frac{1}{2} \sigma  \left(\lambda t+\mu e^{\gamma t}\right)}
		$$
		and
		$$
			R(t)=\frac{\Delta_1}{\beta}e^{\frac{e^{\gamma t} \left(\sigma  \left(\gamma  \mu  \left(2 \sqrt{2}-2 \lambda +2  \gamma  \lambda t+ \gamma  \mu e^{\gamma t}\right)+\kappa  \sigma \right)+4 \gamma ^2 \mu  \log\mathcal A\right)}{8 \gamma  \sigma }}-\frac{\alpha }{\beta }.
		$$
	\item  $\lambda=0$
		\begin{equation}\label{specialSol2}
			F(\zeta) = \Delta_2 e^{\frac{1}{4} \gamma  \zeta ^2+\mathcal C \zeta}
		\end{equation}
		where $\mathcal C$ is a constant and $\Delta_2 = e^{\frac{\left(\gamma  (2-4 \delta )+(1+2\mathcal C)^2-4 \gamma \log\beta\right)}{4\gamma }}$. Using it in combination with \eqref{u2a} we arrive to the invariant solution 
		\begin{multline*}
			u(x,t) = e^{-x/2} \left(\Delta_2 e^{\frac{1}{8} \left(2 (t+x) (2+(t+x) \gamma +4 \mathcal C)+\frac{e^{\gamma t} \left(\kappa  \sigma -4 \sqrt{2} \gamma  \mu  \mathcal C\right)}{\gamma }\right)} -\frac{\alpha}{\beta }\right).
		\end{multline*} 
		Again, using the equivalence transformations and the boundary condition \eqref{barrierCond1} we obtain the similarity solution for Eq.~\eqref{u01} with $f(\zeta)=-\frac{\gamma}{\beta}(\alpha +\beta \zeta)\left(\delta +\log\lvert\alpha+\beta\zeta\rvert\right)$
		\begin{multline*}
			u(x,t) = \\
				\Delta_2 e^{\frac{\sigma ^2 \left(\kappa  \sigma -4 \sqrt{2} \gamma  \mu  \mathcal C\right)e^{t \gamma } +\gamma  \left(\sigma ^2-2 r\right) \left(\sigma  \left(\gamma  \sigma t+2 \sqrt{2} (1+2 \mathcal C)\right)-2 r \gamma t\right)t +4 \gamma ^2 \log^2\lvert x\rvert}{8\gamma  \sigma ^2}}\times\\
					 \lvert x\rvert^{\gamma  \left(\frac{1}{2}-\frac{r}{\sigma ^2}\right)t+\frac{1+2\mathcal C}{\sqrt{2} \sigma }}-\frac{\alpha }{\beta }
		\end{multline*} 
		with
		$$
			H(t) = \mathcal A e^{\frac{1}{2}(2r -\sigma^2)t+\frac{1}{2} \sigma \mu e^{\gamma t}}
		$$
		and
		$$
			R(t)= \Delta_2 e^{\frac{1}{8} \left(e^{2 t \gamma } \gamma  \mu ^2+\frac{e^{t \gamma } \left(2 \sqrt{2} \gamma  \mu +\kappa  \sigma \right)}{\gamma }+\frac{4 \gamma  \log^2\lvert\mathcal A\rvert}{\sigma ^2}\right)} \mathcal A^{\frac{ \gamma  \mu e^{\gamma t}+\sqrt{2} (1+2\mathcal C)}{2 \sigma }}-\frac{\alpha }{\beta }.
		$$
\end{itemize}

\section{Conclusion}

In the present paper a generalization of the celebrated Black--Scholes--Merton equation \eqref{u02}  was proposed and studied under the prism of the modern group analysis or symmetry method. To that end, we harnessed the advantage that the equivalence transformations offer when studying classes of differential equations, the knowledge  of the best representative for this class of equations. This fact substantially simplifies the task of classifying it and obtaining its point symmetries.  

Through this classification interesting cases, from the point of symmetries, arise. Nonlinear equations in general have few or no symmetries so cases that augment the set of symmetries at disposal are like an oasis in the desert. Quite commonly a dynamical system possessing an ample number of symmetries is more probable to relate with a physical system or model a more realistic process. Furthermore, in the case that we wish to study a boundary problem, because of the fact that not all of the symmetries admit the boundary and its condition, some of the symmetries will be excluded. Hence the bigger the set of symmetries the bigger the probability that some will survive the scrutiny of the boundary conditions and give an invariant solution for the problem in its entirety. 

For the equation studied here both sides of this fact were revealed.  The terminal condition was too strict and gave only trivial solutions. On the other hand the boundary condition imposed for the barrier option allowed us to obtain  non trivial invariant solutions, undoubtedly, the arbitrary functions involved in the boundary condition helped in that direction.

The insight provided through the above symmetry analysis might prove practical to anyone looking for a more realistic economic model without departing from the reasoning behind the Black--Scholes-Merton equation that made it such a successful model on the first place. Moreover,  when  one studies more exotic kinds of options that gain ground in the Asian markets that in turn play an ever increasing role in the world market. Last but not least although the Black--Scholes-Merton model is a standard way to price traditional options  it encounters difficulties with exotic ones. The nonlinear variants of the traditional model given here, along with the found analytical solutions, might turn the table in that respect.  We leave to the interested reader the possible economical interpretation and use of the obtained results.

\section*{Acknowledgements} 

Y. Bozhkov would like to thank FAPESP and CNPq, Brasil, for partial financial support. S. Dimas is grateful to FAPESP (Proc. \#2011/05855-9)  for the financial support and IMECC-UNICAMP for their gracious hospitality.

\bibliographystyle{model3-num-names}
\bibliography{Bibliography}

\end{document}